\documentclass[a4paper]{amsart}
\usepackage{amsmath,latexsym,amssymb}
\usepackage{amscd}
\usepackage[utf8]{inputenc}

\usepackage[T1]{fontenc}
\usepackage[osf]{mathpazo}
\usepackage{amssymb} 

\usepackage{eucal}

\usepackage{enumitem}
\setlist[enumerate,1]{label=(\roman*), font=\normalfont}

\newtheorem{theorem}{Theorem}[section]
\newtheorem{lemma}[theorem]{Lemma}
\newtheorem{proposition}[theorem]{Proposition}
\newtheorem{corollary}[theorem]{Corollary}

\newtheorem{definition}[theorem]{Definition}
\newtheorem{remarks}[theorem]{Remarks}

\newtheorem{notation}[theorem]{Notation}

\def\N{{\mathbb{N}}}

\def\R{{\mathbb{R}}}
\def\acl{{\rm acl}}
\def\tp{{\rm tp}}

\def\cl{{\rm cl}}

\def\K{\mathcal{K}}

\def\N{\mathbb{N}}
\def\R{\mathbb{R}}

\def\Psi{(P6) }

\def\B{\mathcal{B}}

\def\x{\bar{x}}
\def\y{\bar{y}}
\def\z{\bar{z}}
\def\a{\bar{a}}
\def\dim{{\rm dim}}
\def\m{\mu}
\def\Def{{\rm Def}}
\def\b{\bar{b}}
\def\c{\bar{c}}

\def\loc{{\rm loc}}

\def\u{\bar{u}}

\def\J{\mathcal{J}}
\def\M{\mathcal{M}}

\def\1{\mathrm{1}}

\def\Ind#1#2{#1\setbox0=\hbox{$#1x$}\kern\wd0\hbox to 0pt{\hss$#1\mid$\hss}
\lower.9\ht0\hbox to 0pt{\hss$#1\smile$\hss}\kern\wd0}
\def\Notind#1#2{#1\setbox0=\hbox{$#1x$}\kern\wd0\hbox to 0pt{\mathchardef
\nn="3236\hss$#1\nn$\kern1.4\wd0\hss}\hbox to 0pt{\hss$#1\mid$\hss}\lower.9\ht0
\hbox to 0pt{\hss$#1\smile$\hss}\kern\wd0}

\begin{document}

\title[Measure and Amalgamation]{{Higher amalgamation properties in measured structures}}

\author{David M. Evans}

\address{%
Department of Mathematics\\
Imperial College London\\
London SW7~2AZ\\
UK.}

\email{david.evans@imperial.ac.uk}

\date{11 July, 2022}

\begin{abstract} Using an infinitary  version of the Hypergraph Removal Lem\-ma due to Towsner, we prove a model-theoretic higher amalgamation result. In particular, we obtain an independent amalgamation property which holds in structures which are measurable in the sense of Macpherson and Steinhorn, but which is not generally true in structures which are supersimple of finite SU-rank. We use this to show that some of  Hrushovski's non-locally-modular, supersimple $\omega$-categorical structures are not MS-measurable.

\end{abstract}
\maketitle

\section{Introduction}

In  \cite{T2}, Towsner  gives an infinitary  version of the Hypergraph Removal Lem\-ma  (quoted as Theorem \ref{HRL} here), stated as a rather general measure-theoretic result. We use this to prove a model-theoretic higher amalgamation result (Theorem \ref{main}), again in the presence of a definable measure. In particular, we obtain an independent amalgamation property (Corollary \ref{23}; quoted below as Corollary \ref{intro23}) which holds in structures which are measurable in the sense of Macpherson and Steinhorn.

The statement of this independent amalgamation property makes no mention of measure and it makes sense in any supersimple structure of finite SU-rank. However, it is not generally true in structures which are supersimple of finite SU-rank. In Theorem \ref{HrCon}, we use a Hrushovski construction to produce a structure which is $\omega$-categorical, supersimple of SU-rank 1 and which does not satisfy the conclusion of Corollary \ref{23}. It follows that this structure is not MS-measurable.

The question of whether any (non-trivial)  $\omega$-categorical Hrushovski construction can be MS-measurable is open, and this is an important special case of the more general question of whether $\omega$-categorical MS-measurable structures are necessarily one-based.  In ongoing work \cite{Marimon}, Marimon has used a different and more generally applicable approach to show that a much wider class of $\omega$-categorical, supersimple Hrushovski constructions are not MS-measurable. It is also unknown whether any of the $\omega$-categorical Hrushovski constructions can be pseudofinite. In Remarks \ref{Psf} we note that, as a by-product of our approach to non-MS-measurability, we obtain information about what coarse pseudofinite dimension would have to be be in such a structure, if it were pseudofinite.

We begin with a rough outline of what we mean by a `higher amalgamation property.' This is adapted to the form of the Towsner's paper, so is slightly different from other presentations (for example in \cite{Hr:groupoids}). 

Suppose $L$ is a first-order language and $\M$ is an $L$-structure with domain $M$ and $C \subseteq M$. Let $T$ denote the theory of $\M$. We will assume that $\M$ is `large' (for example $\aleph_1$-saturated, if  $L$ is countable) and $C$ has smaller cardinality than that of $M$. Suppose $n \geq 2$ is a natural number. In an \textit{$n$-amalgamation problem} over $C$ we are looking for an $n$-tuple 
$\b = (b_1,\ldots, b_n)$ which satisfies certain constraints on subtuples $\b_I = (b_i : i \in I)$ with $I \subseteq [n] = \{1,\ldots, n\}$ of size $n-1$. The constrains should be in terms of the parameters $C$, say in the form of satisfying a type, or partial type, $\Phi_I(\x_I)$ over $C$. Here, $\x = (x_1,\ldots, x_n)$ is an $n$-tuple of variables and $x_I = (x_i : i \in I)$. So, subject to reasonable   compatibility requirements such as $\Phi_I(\x_I)$ and $\Phi_J(\x_J)$ having the same restriction to $\x_{I\cap J}$, we are looking for a solution $\b \models \bigwedge_I \Phi_I(\x_I)$, or, in terms of the sets $A_I = \{\a \in M^n : \M \models \Phi_I(\a)\}$, an element of $\bigcap_I A_I$. If  the $\Phi_I$ are complete types over $C$, we might refer to this as a \textit{type-amalgamation} problem. 

There are well-known variations on this. If $\M$ carries a notion of independence (or  dimension on definable sets) then in an \textit{independent} $n$-amalgamation problem over $C$, we are also looking for the $b_i$ to be independent over $C$. Of course, in this case, the individual constraints $\Phi_I(\x_I)$ should have solutions which are independent over $C$. For example, if $T$ is stable, then for all $n$, any independent type-amalgamation problem over a model (with $n$ complete types over the model) has a solution. If $T$ is simple, then this is true for $n = 2, 3$ (the case $n = 3$ is of course the Independence Theorem of Kim and Pillay). However, there are examples of supersimple theories of finite $SU$-rank which do not have independent $4$-amalgamation over a model.

Our main result, Theorem \ref{main}, is an $n$-amalgamation property which holds in a general context where the set of $n$-tuples from which we are looking for a solution carries a well-behaved probability measure (see Section \ref{setup} for a precise statement). The general form of the statement is that we assume there our $n$-amalgamation problem has `degenerate' solutions $\b = (b_1,\ldots, b_n)$, where the $b_i$ are interalgebraic over $C$. The conclusion is that the set of all solutions is of positive measure (and in particular, there are solutions where the $b_i$ are not interalgebraic). Of course, for this to work, we need to ensure that there are enough solutions to the $\Phi_I$: in the above notation, we require that the measure of $A_I$ is positive, for each $(n-1)$-set $I$.

If $\M$ is an MS-measurable structure (see Section \ref{MSsection} for definitions and background) there is a strong interaction between dimension and measure. The structure  $\M$ is supersimple of finite $SU$-rank and each definable subset has an associated dimension (which can be taken as $SU$-rank for the purposes of this Introduction). Each definable set also carries a (definable) probability measure on its definable subsets with the property that a subset has positive measure if and only if it has the same dimension as the ambient definable set. 

 From Theorem \ref{main} we obtain the following independent amalgamation result (Corollary \ref{23}), which holds in any MS-measurable $M$. 
\begin{corollary} \label{intro23} Suppose $\M$ is an MS-measurable structure and $S_1,\ldots, S_n$ are infinite $C$-definable sets, for some finite $C \subset M$. Let $S = S_1 \times \ldots \times S_n$ and for $I \subset [n] = \{1,\ldots, n\}$, let $\pi_I : S \to \prod_{i\in I} S_i$ be the projection map. Suppose  $E \subseteq S$ is a $C$-definable subset such that:
\begin{enumerate}
\item[(a)] if $I \subset [n]$ and $\vert I \vert = n-1$, then $\dim(\pi_I(E)) = \sum_{i \in I} \dim(S_i)$ , and
\item[(b)] if $(b_1,\ldots, b_n) \in E$, then $b_i \in \acl(C \cup\{ b_j : j \neq i\})$.
\end{enumerate}
 Then 
\[\dim(\{\b \in S : \pi_I(\b)\in \pi_I(E) \mbox{ for all } I \mbox{ with } \vert I \vert = n-1\}) = \dim(S).\]
\end{corollary}

Note that this does not tell us anything if $\M$ has trivial algebraic closure. Note also that it does not refer to the measure, so it makes sense in any supersimple theory (more properly, any $S_1$-theory) of finite $SU$-rank. In Section \ref{App} we give an example of a supersimple structure of $SU$-rank 1 which does not satisfy the above result: so we have an independent amalgamation result which holds in MS-measurable structures, but which  is not  generally true in  finite rank supersimple structures.

This paper is a revised version of some unpublished notes written in 2011-12. The original version made use of Towsner's unpublished article \cite{T} and proved Theorem \ref{main} under a stronger assumptions on the definability of the measure and the behaviour of the measure under projection maps with finite fibres. In 2019, I sent a copy of the notes to Ehud Hrushovski who observed that that these assumptions could be weakened. He also gave examples of additional contexts in which the weaker assumptions would hold: cf.  Section \ref{NIP} here. 

Towsner's published paper \cite{T2} contains a reworking of \cite{T} which involves a weaker assumption on the definability of the measure. In revising the original notes, I have therefore rewritten the proof of Theorem \ref{main} to follow the approach and notation of \cite{T2}. 

The structure of the paper is as follows. In Section \ref{setup} we give the necessary notation and background to state Towsner's version of the Hypergraph Removal Lemma from \cite{T2}. In Section \ref{Main_section} we deduce the main result Theorem \ref{main} from this. Our  result is related to a standard deduction of Szemer\'edi's Theorem from the Hypergraph Removal Lemma: we make this explicit in Section \ref{Szem}. In Section \ref{MSsection}, we discuss MS-measurability and prove Corollary \ref{23}, stated above. Additional examples in NIP theories are mentioned briefly in Section \ref{NIP}. In Section 4, we discuss the $\omega$-categorical Hrushovski constructions and their relationship to various open questions around MS-measurable $\omega$-categorical structures. The main result of the section is Theorem \ref{HrCon}, where we construct an $\omega$-categorical structure which is of $SU$-rank 1 and which does not satisfy the amalgamation property in Corollary \ref{23}.

{\sc Acknowledgements:\/} The Author is very much indebted to Ehud Hrushovski for his comments and suggestions on a previous version of this paper.

\section{An amalgamation theorem for measured structures}

\subsection{Measured structures} \label{setup}

The following set-up is taken from Towsner's paper \cite{T2}. Chapter 1 of \cite{K} is a convenient reference for the basic measure theory we need.

We work with a structure $\M$ with domain $M$. The following notation is introduced in Section 2 of \cite{T2}. If $V$ is a finite set of indices, then a $V$-tuple from $M$ is a function $\a_V : V \to M$ and we denote the set of these by $M^V$. A $V$-tuple of variables will generally be denoted by $\x_V$. If $I \subseteq V$ then $\a_I \in M^I$ is the restriction of this to $I$. If $U, W$ are disjoint sets, we write $\a_U \cup \a_W$ for the $U \cup W$-tuple extending $\a_U$ and $\a_W$. If $B \subseteq M^{U\cup W}$ and $\a_W \in M^W$, then $B(\a_W)$ denotes the fibre (or `slice') $\{\a_U \in M^U : \a_U \cup \a_W \in B\}$.

In what follows, $V$ is a fixed finite set of indices $V = \{1,\ldots, n\} = [n]$ for some $n \in \N$. We often denote $\a_V$ or $\x_V$ simply by $\a$, $\x$, dropping the reference to $V$.

\begin{definition} \label{BA}\rm (\cite{T2}, Definition 4.1) Suppose that for each $U \subseteq V$  we have a Boolean algebra $\B^0_U$ of subsets of $M^U$ such that:
\begin{itemize}
\item $\emptyset \in \B^0_U$;
\item $\B^0_U \times \B^0_W \subseteq \B^0_{U\cup W}$ for disjoint $U, W \subseteq V$;
\item If $U, W \subseteq V$ are disjoint, $\a_W \in M^W$ and $B \in \B^0_{U\cup W}$, then $B(\a_W) \in \B^0_U$.
\end{itemize}
For $I \subseteq V$ we define $\B^0_{V, I}$ to be the Boolean algebra generated by subsets $\{ \a_V \in M^V : \a_I \in B\}$, where $B \in \B^0_I$. 

In all cases we will drop the superscript $0$ to indicate the $\sigma$-algebra generated by the Boolean algebra.
\end{definition}

The main result we need from \cite{T2} is Theorem \ref{HRL} below. When we use this,   $\B^0_V$ will  consist of the parameter-definable subsets of $M^V$, so the reader may assume this from now on. We then refer to the elements of $\B_V$ as Borel sets. If $X \subseteq M$, then $\B^0_V(X)$ will denote the $X$-definable subsets of $M^V$, and we use a corresponding variation in the notation for the algebras introduced above. We will assume sufficient saturation, so that it makes sense to identify a formula defining a Borel set with its solution set in $M$. In particular, if the language is countable, we assume that $\M$ is $\aleph_1$-saturated. If the model is multi-sorted, then we can restrict each variable to having values in a particular sort.

Suppose, with the above notation,  that $\nu = \nu^V$ is a probability measure on $(M^V; \B_V)$. If $I \subseteq V$, then  let $\nu^I$ denote the push-forward measure on $(M^I; \B_I)$. So for $A \in \B_I$, we have  $\nu^I(A) = \nu(\pi_I^{-1}(A))$, where $\pi_I : M^V \to M^I$ is the projection map.

Recall that if $\nu$ is a probability measure on a $\sigma$-algebra $\B$ of subsets of a set $N$, then $L^\infty(\B)$ denotes the space of $\B$-measurable functions $N \to \R$ which are essentially bounded, that is, are bounded outside a set of measure $0$.

Henceforth, we shall assume that the following conditions on $\nu$ hold.

\begin{itemize}
\item \textbf{(Definability)} For all $J \subseteq V$ and $B \in \B_V$, the function $x_{J} \mapsto \nu^{V\setminus J}(B(x_{J}))$ is $\B_{J}$-measurable.
\item \textbf{(Fubini)} Suppose $J \subseteq V$ and  $f \in L^\infty(\nu^V)$. Then $\int f d\nu^V = \int\int f d\nu^J d\nu^{V\setminus J}$. 
\end{itemize}

\smallskip

\begin{remarks}\rm \label{Rem22}

(1) It would be more correct to refer to the Definability condition as `Borel definability', but we will not do this.\newline
(2) It suffices to check that the Definability property holds for all $B \in \B_V^0$, as the set of elements of $\B_V$ for which it holds is a $\sigma$-subalgebra.\newline
(3) The Definability property is a weaker requirement than asking that $\nu$ is invariant (over the emptyset, or a small submodel).\newline
(4) The Definability property implies that, in the statement of the Fubini condition, the map $$\x_{V\setminus J} \mapsto \int f(\x_J\x_{V\setminus J}) d\nu^J(\x_J)$$ is $\B_{V\setminus J}$-measurable for almost all $\x_{V\setminus J} \in M^{V\setminus J}$. This is a standard argument using approximation by indicator functions of sets in $\B_V$. The same sort of argument shows that it suffices to check the Fubini condition in the case where  $f$ is the indicator function $\1_B$ of a set $B \in \B^0_V$. 
\end{remarks}

The following is Towsner's infinitary analogue of the Hypergraph Removal Lemma. We refer to \cite{T, T2} for a discussion of the origins of the proof and the finitary versions of this. The statement follows by combining Theorem 5.3 and Lemma 5.4 of \cite{T2}. Theorem 5.3 of \cite{T2} holds under weaker conditions than the Fubini property (involving the notion of \textit{$J$-regularity} of $\nu^V$), but we will not make use of this. Lemma 5.4 of \cite{T2} states that the Definability and Fubini conditions imply $J$-regularity of $\nu^V$for all $J \subseteq V$. 

\begin{theorem} \label{HRL} (\cite{T2}, Theorem 5.3) Suppose $\M$ is sufficiently saturated and $\B^0_V$ consists of the definable subsets of $M^V$. Suppose  $\nu^V$ is a probability measure on $\B^V$ which satisfies the Definability and Fubini conditions. Let $k < n = \vert V \vert$ and $\J = [V]^k$, the set of $k$-subsets from $V$. 

Let $A_I \in \B_{V, I}$ for $I \in \J$. Suppose there is $\delta > 0$ such that
 whenever $B_I \in \B_{V,I}^0$ are such that $\nu^V(A_I\setminus B_I) < \delta$, then $\bigcap_{I \in \J} B_I \neq \emptyset$.

Then $\nu^V(\bigcap_{I \in \J} A_I) > 0.$  $\Box$
\end{theorem}

\subsection{A Model-theoretic Corollary} \label{Main_section}
In the following, we give  model-theoretic conditions  which allows us to verify the hypotheses in Theorem \ref{HRL}. The set-up is:
\begin{itemize}
\item $\M$ is an $\aleph_1$-saturated structure in a  countable language $L$;
\item $V = \{1,\ldots, n\}$ is a set of indices (each associated to a particular sort); we let $J = \{1,\ldots, n-1\} \subseteq V$ and $\J$ the set of $(n-1)$-subsets of $V$;
\item for each $I \subseteq V$, $\B_I^0$ is the Boolean algebra  of $M$-definable subsets of $M^I$;
\item $\nu = \nu^V$ is a probability measure  on $\B_V$ which satisfies the Definability and Fubini conditions.
\end{itemize}

For $I \subseteq V$ let $\pi_I$ denote the projection map $M^V \to M^I$ and denote by $\nu^I$ the push-forward measure induced on $\B_I$ by $\nu$. Each $\nu^I$ also satisfies the corresponding Definability and Fubini properties.

\begin{theorem} \label{main} With the above notation and assumptions,  suppose $E \in \B_V$  is such that:
\begin{itemize}
\item[(a)] $\nu^J(\pi_J(E)) > 0$;
\item[(b)] there is $l \in \N$ such that for all  $I \in \J$ and $\a \in M^I$, we have that $\pi_I^{-1}(\a)\cap E$ has at most $l$ elements;
\item[(c)]  there is $k > 0$  such that if $F \in \B^0_V$, then $\nu^J(\pi_J(F\cap E)) \leq k\nu^I(\pi_I(F \cap E))$ for all $I \in \J$. 
\end{itemize}
Then $\nu^V(\{ \b \in M^V: \pi_I(\b) \in \pi_I(E) \mbox{ for all  } I \in \J\}) > 0.$
\end{theorem}

\begin{remarks}\rm \label{RN} We make some comments about the conditions on $E$. By the second condition, we should not expect that $\nu(E) > 0$. However, suppose that we also have a measure $\lambda$ on the definable subsets of $E$ with $\lambda(E) > 0$ and $r, s >0$ such that for all $F \in \B_0^V$ and $I \neq J$ we have 
$$r \nu^J(\pi_J(F\cap E)) \leq \lambda(F\cap E) \leq s \nu^I(\pi_I(F\cap E)).$$
Then $\nu^J(\pi(F\cap E)) \leq \frac{s}{r} \nu^I(\pi_I(F\cap E))$, so the third condition holds.

In general, without assuming the existence of such a $\lambda$, we can define a measure $\nu_I^J$ on $\pi_J(E)$ by setting $\nu_I^J(X) = \nu^I(\pi_I(\pi_J^{-1}(X)\cap E))$. Condition (c) implies that $\nu^J$ is absolutely continuous with respect to $\nu_I^J$ and $k$ is a bound on the Radon-Nikod\'ym derivative. 
\end{remarks}

Before proving Theorem \ref{main} we note the following lemmas.

\begin{lemma}\label{lem} With the notation as in Theorem \ref{main}, suppose  $F \subseteq E$ is a countable intersection of sets in $\B^0_V$ with $E$. Then:
\begin{enumerate}
\item[(1)] $\nu^J(\pi_J(F)) \leq k\nu^I(\pi_I(F))$.
\item[(2)] If $J\neq I \in \J$ and $C \in \B^0_I$ then 
\[\nu^J(\pi_J(F\setminus\pi_I^{-1}(C)) \geq \nu^J(\pi_J(F)) - k\nu^I(C\cap \pi_I(F)).\]
\item[(3)] If $J\neq I \in \J$ and $B \in \B^0_I$, then 
\[\nu^J(\pi_J(F\cap \pi_I^{-1}(B)) \geq \nu^J(\pi_J(F)) - k\nu^I(\pi_I(F)\setminus B).\]
\end{enumerate}
\end{lemma}

\begin{proof} (1) Write $F = E \cap \bigcap_{i<\omega} F_i$ where each $F_i$ is in $\B^0_V$. We can assume that $F_i \supseteq F_{i+1}$. Then $\aleph_1$-saturation implies $\pi_J(F) = \bigcap_{i<\omega} \pi_J(E \cap F_i)$ and $\nu^J(\pi_J(F)) = \inf(\nu^J(\pi_J(E \cap F_i)): i<\omega)$.  By assumption on $E$, we have $\nu^J(\pi_J(E \cap F_i)) \leq k\nu^I(\pi_I(E \cap F_i))$ for each $i$; taking the limit gives what we require.

(2) By (1) we have
\[\nu^J(\pi_J(\pi_I^{-1}(C)\cap F)) \leq k\nu^I(\pi_I(F)\cap C).\]
Of course, $\pi_J(F) = \pi_J(\pi_I^{-1}(C)\cap F) \cup \pi_J(F\setminus\pi_I^{-1}(C))$, so 
\[\nu^J(\pi_J(F)) \leq \nu^J(\pi_J(\pi_I^{-1}(C)\cap F))+ \nu^J(\pi_J(F\setminus\pi_I^{-1}(C))).\]
Putting these together gives the required result.

(3) Apply (2), taking $C$ to be the complement of $B$.
\end{proof}

\smallskip

\begin{lemma} \label{lemma2} Suppose $E \in \B_V$ with $\nu^J(\pi_J(E)) >0$ and, for all $\a \in E$, we have that  $\pi_J^{-1}(\pi_J(\a)) \cap E$ has at most $l$ elements.  Let $X \subseteq M$ be a countable set over which $E$ is definable.
\begin{enumerate}
\item[(1)] There is some $F \in \B_V(X)$ with $F \subseteq E$, a natural number $r$ and an $L(X)$-formula $\psi(\x)$ such that $\nu^J(\pi_J(F)) > 0$, and if $\a \in F$, then $\psi(\a_J, x_n)$ isolates $\tp^\M(a_n/\a_J X)$, and this type has precisely $r$ solutions in $\M$. The set $F$ can be taken to be a countable intersection of sets in $\B_V^0(X)$ with $E$.
\item[(2)] If $X$ is chosen so that $r$ in (1) is minimal, then for countable $Y \supseteq X$ and for almost all $\a_J \in \pi_J(F)$, if $(\a_J,a_n) \in F$, then $\psi(\a_J,x_n)$ isolates $\tp^\M(a_n/ \a_JY)$ (and therefore this type has the same solutions as $\tp^\M(a_n/\a_JX)$).
\end{enumerate}
\end{lemma}

\begin{proof} (1) For each $V$-variable formula $\psi(\x) \in L(X)$, and $r \leq l$ consider the set $E_{\psi, r}$ consisting of those $(a_1,\ldots, a_n) \in E$ such that the formula $\psi(a_1,\ldots, a_{n-1}, x_n)$ isolates $\tp(a_n/a_1,\ldots, a_{n-1}, X)$, and this type has $r$ solutions in $\M$. As $E$ is defined over $X$, all of these solutions lie in $E$. Note that $E_{\psi, r}$ is defined by the conjunction of $E$ and:

\begin{multline*}
\bigwedge_{\chi\in L(X)} \psi(x_1,\ldots, x_n) \wedge (\exists^{=r} x_n)\psi(x_1,\ldots, x_n) \wedge \\ (\forall y)(\psi(x_1,\ldots, x_{n-1}, y) \to (\chi(x_1,\ldots, x_n) \leftrightarrow \chi(x_1,\ldots, x_{n-1}, y))),\end{multline*}
so is in $\B_V(X)$. Moreover, $\bigcup_{\psi;\, r\leq l} E_{\psi,r} = E$ (by the algebraicity). So as this is a countable union, there are $\psi$ and $r \leq l$ with $\nu^{J}(\pi_J(E_{\psi,r})) >0$. Then $F = E_{\psi, r}$ has the required properties.

(2) Let $Y \supseteq X$ be a countable subset of $M$ and consider 
\[E' = \{ \a \in F : \psi(\a_J,x_n) \mbox{ does not isolate } \tp(a_n/\a_JY)\}.\]
As in (1), we have $E' \in \B_V(Y)$. Suppose for a contradiction that $\nu^J(\pi_J(E')) > 0$. Applying (1) we obtain $F' \in \B_V(Y)$ with $F' \subseteq E'$ and $\nu^J(\pi_J(F')) > 0$, some $r' \in \N$ and an $L(Y)$-formula $\psi'$ such that for all $\a \in F'$, $\psi'(\a_J, x_n)$ isolates $\tp(a_n/\a_JY)$ and the latter has $r'$ solutions. By definition of $E'$ we have $r' < r$ and this contradicts the choice of $r$. Thus $\nu^J(\pi_J(E')) = 0$ and the result follows.
\end{proof}

We now prove  Theorem \ref{main}.

\begin{proof}[Proof of \ref{main}]

From Lemma \ref{lemma2} (2), there is a countable subset  $X $ of $M$ containing  the parameters for $E$ and a countable intersection $F$ of $X$-definable sets with $E$ such that \begin{itemize}
\item $\nu^{J}(\pi_J(F)) > 0$; 
\item if  $(a_1,\ldots, a_{n-1}, a_n)$, $(a_1,\ldots, a_{n-1}, a_n') \in F$, then $\tp^\M(a_n/ a_1,\ldots, a_{n-1}, X) = \tp^\M(a_n'/ a_1,\ldots, a_{n-1}, X)$;
\item if $Y \supseteq X$ is countable, then for almost all $\a \in F$, $\tp^\M(a_n/\a_JX)$ and $\tp^\M(a_n/\a_JY)$ have the same solutions.
\end{itemize}

 To see the second point here, note that the two types are isolated by the same formula, so must be equal. The other points are directly from Lemma \ref{lemma2}.
\smallskip

For $I \in \J$, let $A_I = \pi_I^{-1}(\pi_I(F))$. So of course, $A_I \in \B_{V,I}$
 and  $F \subseteq \bigcap_{I \in \J} A_I$. We verify that the hypotheses of Theorem \ref{HRL} hold.
 
 Let $\delta > 0$ (to be fixed later) and $B_I \in \B^0_{V,I}$ with $\nu^{V}(A_I\setminus B_I) < \delta$.  Note that  $\nu^V(A_I) = \nu^{I}(\pi_I(F))$ and similarly $\nu^{I}(\pi_I(F)\setminus \pi_I(B_I)) = \nu^V(A_I\setminus B_I)$. Therefore, with $k$ as in condition (c) of Theorem \ref{main} and $I \neq J$, Lemma \ref{lem} (3) gives:
\begin{multline*} \nu^{J}(\pi_J(F\cap B_I)) \geq \nu^{J}(\pi_J(F)) - k\nu^{I}(\pi_I(F)\setminus \pi_I(B_I)) > \nu^{J}(\pi_J(F)) - k\delta.\end{multline*}
This also holds with $I = J$, as $k \geq 1$. 

Now let $\eta = \nu^{J}(\pi_J(F))$ (so $\eta>0$, by choice of $F$) and  $\delta = \eta/2kn$. We obtain, for all $I \in \J$:
\[\nu^{J}(\pi_J(F\cap B_I)) \geq (1 -1/2n)\eta.\]

The  measure of the union of the complements of the sets $\pi_J(F \cap B_I)$ in $\pi_J(F)$ is  therefore at most $\eta/2$, and so
\[\nu^{J}(\bigcap_{I \in \J}\pi_J(F\cap B_I)) \geq \eta/2.\]
Let $Y$ be the union of $X$ and the parameter sets of the $B_I$. Then  we can find $\b_J = (b_1,\ldots, b_{n-1}) \in \bigcap_I \pi_J(F\cap B_I)$ such that if $(\b_J, b_n)$ and $(\b_J, b_n') \in F$, then they have the same type over $Y$. Indeed, almost all $\b_J \in \pi_J(F)$ have this property, by our conditions on $F$. 

Take $b_n \in M$ with $\b=(b_1,\ldots, b_{n-1}, b_n) \in F$. We show that  $(b_1,\ldots, b_n) \in \bigcap_I B_I$, and thus the hypotheses of Theorem \ref{HRL} hold.

\smallskip

Clearly $\b \in B_J$. Take $I \neq J$. As $(b_1,\ldots, b_{n-1}) \in \pi_J(F\cap B_I)$ there is $(b_1', b_2',\ldots, b_n') \in F \cap B_I$ with $(b_1',\ldots, b_{n-1}') = (b_1,\ldots, b_{n-1})$. So $(b_1,\ldots, b_{n-1}, b_n), (b_1,\ldots, b_{n-1}, b_n') \in F$, and therefore $b_n$ and $b_n'$ have the same type over $Y \cup\{b_1, \ldots, b_{n-1}\}$. As $B_I$ is defined over $Y$ and $(b_1,\ldots, b_{n-1}, b_n') \in B_I$, it follows that  $(b_1,\ldots, b_{n-1}, b_n) \in B_I$, as required.

We have shown that $\bigcap_I B_I \neq \emptyset$, so Theorem \ref{HRL} applies to give that $\nu(\bigcap_{I} A_I) > 0$. As $\bigcap_{I} A_I \subseteq \{ \b \in \M^V : \pi_I(\b) \in \pi_I(E) \mbox{ for all } I \in \J\}$, we have the result.
\end{proof}

\section{Examples and applications}

\subsection{Pseudofinite structures and Szemer\'edi's Therorem}\label{Szem} In \cite{T} (and Section 5 of \cite{T2}), the structure $\M$ is an ultraproduct of finite structures $(F_i : i < \omega)$ and the measures arise by taking the standard part of ultraproducts of normalised counting measures on the $F_i$. The original language is enriched to ensure Definability of the measure. The Fubini property then follows as we are dealing with counting measures.  

In Section 2 of \cite{T}, Szemer\'edi's Theorem is deduced from Theorem \ref{HRL} in the following way (we do not give the details: the point is to explain where the statement of Theorem \ref{main} comes from).  The original language is that of abelian groups (written additively) and there is a predicate $A(.)$ for a subset of the group. Each $F_i$ is cyclic of prime order (increasing with $i$) and $A[F_i]$ is some subset of $F_i$. Denoting the ultraproduct (in the enriched language) by $G$, the main assumption is that the measure of $A[G]$ is strictly positive.

So $(G, +)$ is a torsion-free, divisible abelian group and, if $n \in \N$ and $n \geq 3$, we have a definable measure $\nu^n$ on the definable subsets of $G^n$ which satisfies the hypotheses of Theorem \ref{main}. The measure is invariant under definable bijections (in particular, under translations and taking $i$-th roots).   Let 
\[E = \{ (x_1,\ldots, x_{n-1}, \sum_{i < n} x_i) :  \sum_{i < n} ix_i \in A\}.\]

This is definable and, in the notation of Theorem \ref{main}, $\pi_J(E) = \nu^1(A) > 0$ (using the divisibility of $G$ and invariance of the measure under definable bijections). The projection maps $\pi_I$ (with $\vert I \vert = n-1$) are injective on $E$ and thus the remaining two conditions in Theorem \ref{main} hold (with $k = l = 1$). 

By Theorem \ref{main}, there is therefore some $\b = (b_1,\ldots, b_n) \in G^n$ such that $\pi_I(\b) \in \pi_I(E)$ for all $I$ of size $n-1$ and, by positivity of the measure, we can take $d = b_n - \sum_{i < n}b_j$ to be non-zero. The definition of $E$ means that if we set $a = \sum_{i < n}ib_i$, then $a, a+d, \ldots, a + (n-1)d \in A$. So as $d \neq 0$, we have an $n$-term arithmetic progression in $A$.

\subsection{An amalgamation result in MS-measurable structures}\label{MSsection}

The notion of a \textit{measurable structure} is introduced in the paper \cite{MS} by Macpherson and Steinhorn, following on from observations of Chatzidakis, van den Dries and Macintyre in \cite{CDM}. A survey by Elwes and Macpherson of results and open questions is given in \cite{EM}. Following \cite{KP}, we refer to this notion as \textit{MS-measurability}.

We  recall the definition of MS-measurability from (\cite{MS}; Definition 5.1). For a (first-order) $L$-structure $\M$ we denote by $\Def(\M)$ the collection of all non-empty parameter definable subsets of $M^n$ (for all $n \geq 1$).

\begin{definition}\label{MSmeas} \rm A  structure $\M$ is \textit{MS-measurable} if there is a \textit{dimension - measure} function $h : \Def(\M) \to \N \times \R^{>0}$ satisfying the following, where we write $h(X) = (\dim(X), \m(X))$:
\begin{enumerate}
\item[(i)] If $X$ is finite (and non-empty) then $h(X) = (0, \vert X \vert)$;
\item[(ii)] For every formula $\phi(\x, \y)$ there is a finite set $D_\phi \subseteq \N \times \R^{>0}$ of possible vaules for  $h(\phi(\x, \a))$ (with $\a \in M^n$) and for each such value, the set of $\a$ giving this value is $0$-definable;
\item[(iii)] (Fubini property) Suppose $X, Y \in \Def(\M)$ and $f : X \to Y$ is a definable surjection. By (ii), $Y$ can be partitioned into disjoint definable sets $Y_1,\ldots, Y_r$ such that  $h(f^{-1}(y))$ is constant, equal to $(d_i, m_i)$, for $y \in Y_i$. Let $h(Y_i) = (e_i, n_i)$.  Let $c$ be the maximum of $d_i+e_i$ and suppose this is attained for $i = 1,\ldots, s$. Then $h(X) = (c, m_1n_1+\cdots+m_s n_s)$. 
\end{enumerate}
\end{definition}

In the above, $\dim(X)$ is the \textit{dimension} and $\m(X)$ the \textit{measure} of $X$.  Clearly we can normalise and assume   that $\mu(M) = 1$. We also extend the definition so that $\mu(\emptyset) = 0$. Note that MS-measurability is  a property of the theory of $M$, so any elementary extension or submodel of $\M$ is MS-measurable if $\M$ is. As observed in (\cite{MS}; Remark 5.2), measurability implies supersimplicity and dimension dominates $D$-rank, but is not necessarily equal to it. By (\cite{MS}, Proposition 5.10), the dimension - measure function extends to definable subsets of $\M^{eq}$. 

We suppose (for convenience) that $L$ is countable and suppose that $\M$ is an $\aleph_1$-saturated MS-measurable structure with dimension-measure function $h = (\dim, \m)$. Let $S \in \Def(\M)$ be infinite and let $\B^0_S$ denote the set of definable subsets of $S$. For  $D \in \B^0_S$ we define:
\[\nu^S(D) = \left\{ \begin{array}{ll} \mu(D)/\mu(S) & \mbox{ if $\dim(D) = \dim(S)$}\\
                                                            0       & \mbox{ otherwise}\end{array}\right. .\]

If $X_1, X_2 \in \B^0_S$ are disjoint, then (iii) of Definition \ref{MSmeas} (with $Y$ a two-point set) shows that $\nu^S(X_1\cup X_2) = \nu^S(X_1)+\nu^S(X_2)$. So $\nu^S$ is a finitely-additive probability measure on $\B^0_S$ and it therefore extends uniquely to a probability measure on $\B_S$, which we will also denote by $\nu^S$. 

Now suppose that $S_1,\ldots, S_n \in \Def(\M)$ are infinite  and $S = S_1 \times \ldots \times S_n$. If $I \subseteq V = \{1,\ldots, n\}$, let $S_I$ be the product of the $S_i$ for $i \in I$. As previously, $\pi_I : S \to S_I$ is the projection map. By considering this, (iii) in Definition \ref{MSmeas} gives that $\dim(S) = \dim(S_I) + \dim(S_{V\setminus I})$ and $\mu(S) = \mu(S_I)\mu(S_{V\setminus I})$. 

Let $\nu = \nu^V = 
\nu^S$. If $I \subseteq V$, then the push-forward measure $\nu^I$ on $\B_{S_I}$ obtained from $\nu$ and $\pi_I$ is equal to $\nu^{S_I}$, as defined above. Indeed, it suffices to check this for $D \in \B^0_{S_I}$. If $\dim(D) = \dim(S_I)$, then 
\[\nu^I(D) = \nu^V(D \times S_{V\setminus I}) =  \mu(D \times S_{V\setminus I})/ \mu(S) = \mu(D)\mu(S_{V\setminus I})/\mu(S) = \mu(D) / \mu(S_I)\]
and this is equal to  $\nu^{S_I}(D)$. If $\dim(D) < \dim(S_I)$, then $\dim(D \times S_{V\setminus I}) < \dim(S)$, so both $\nu^I(D)$ and $\nu^{S_I}(D)$ are zero. 

The Definability and Fubini properties given in Section \ref{setup} hold for the $\nu^I$, using (ii) and (iii) of Definition \ref{MSmeas} (cf. Remarks \ref{Rem22}).

From Theorem  \ref{main} we obtain the following, which can be seen as a weak form of independent $n$-amalgamation:

\begin{corollary}\label{23} Suppose $\M$ is an MS-measurable structure and $S_1,\ldots, S_n \in \Def(\M)$ are infinite and defined over a finite set $C \subset M$. Let $S = S_1 \times \ldots \times S_n$ and suppose  $E \subseteq S$ is a $C$-definable subset such that:
\begin{enumerate}
\item[(a)] $\dim(\pi_I(E)) = \sum_{i \in I} \dim(S_i)$ for all $I \in [n]^{n-1}$, and
\item[(b)] If $(b_1,\ldots, b_n) \in E$, then $b_i \in \acl(C \cup\{ b_j : j \neq i\})$. 
\end{enumerate}
 Then 
\[\dim(\{\b \in S : \pi_I(\b)\in \pi_I(E) \mbox{ for all } I \in [n]^{n-1}\}) = \dim(S).\]
\end{corollary}

\begin{remarks}\rm  Assumptions (a) and (b)  in Corollary \ref{23} imply that the $S_i$ have the same dimension. Indeed, $\sum_{j < n}\dim(S_j) = \dim(\pi_J(E)) = \dim(E) = \dim(\pi_I(E)) = \sum_{i \in I} \dim(S_i)$ for all $I \in [n]^{n-1}$. So $\dim(S_j) = \dim(S_n)$ for all $j < n$. 
\end{remarks}

We now prove Corollary \ref{23}.
                                                          
\begin{proof}
We may assume that $\M$ is $\aleph_1$-saturated. We check that the three conditions of Theorem \ref{main} hold. 

By (a), $\dim(\pi_J(E)) = \dim(S_J)$, so  $\nu^J(\pi_J(E)) = \mu(\pi_J(E))/\mu(S_J) >0$. 

As $E$ is definable, by compactness we have a uniform bound $l$ on the algebraicity in Assumption (b). This  gives the second condition required by Theorem \ref{main}. 

Suppose $I \in [n]^{n-1}$. The restriction of the projection map $E \to \pi_I(E)$ has finite fibres, of size at most $l$. Suppose $X \subseteq E$ is definable. If we decompose $\pi_I(X)$ according to the size of the fibres $X \to \pi_I(X)$ and apply (i) and (iii) of \ref{MSmeas}, we obtain 
\[ \mu(\pi_I(X)) \leq \mu(X) \leq l \mu(\pi_I(X)).\]
Thus 
\[ \mu(\pi_J(X)) \leq \mu(X) \leq l \mu(\pi_I(X)).\]
If $\dim(X) = \dim(E)$, then $\dim(\pi_I(X)) = \dim(\pi_I(E)) = \dim(S_I)$ (by (a)) and we obtain
\[\nu^J(\pi_J(X)) \leq l\frac{\mu(S_I)}{\mu(S_J)}\nu^I(\pi_I(X)).\]
If $\dim(X) < \dim(E)$ then the inequality  is also true, as both sides are zero. So we have the third condition required by Theorem \ref{main}.

So by Theorem  \ref{main} 
\[\nu^V(\{\b \in S : \pi_I(\b)\in \pi_I(E) \mbox{ for all } I \in [n]^{n-1}\}) >0\]
and the conclusion follows. 
\end{proof}

\subsection{Further examples}\label{NIP}  If $Th(\M)$ is NIP, then \textit{generically stable} measures (see \cite{HPS}, or \cite{Simon}) provide examples of measures satisfying the Definability and Fubini conditions. More precisely, suppose $\nu_{x_1},\ldots, \nu_{x_n}$ are generically stable measures for $\M$ (in the indicated variables) and let $\nu^V = \nu_{x_1}\otimes \cdots  \otimes\nu_{x_n}$. Then $\nu^V$ has the Definability and Fubini properties, and so Theorem \ref{HRL} and Theorem \ref{main} hold. It would be interesting to know whether either of these results  is saying something new, or at least non-trivial, in this context.

\section{MS-measurability and the Hrushovski construction}\label{App}

In (\cite{EM}, Definition 3.13), a complete theory is defined to be \textit{unimodular} if in any model $\M$, whenever $f_i : X \to Y$ are definable $k_i$-to-1 surjections in $\M^{eq}$ (for $i = 1,2$), then $k_1 = k_2$. (See \cite{KP} for comments on this; in particular why it should more properly be termed weak unimodularlity.) An  MS-measurable structure is necessarily superstable of finite $SU$-rank  and unimodular, and  Question 7 of \cite{EM} asks whether the converse holds. Unimodularity is implied by $\omega$-categoricity (\cite{EM}, Proposition 3.16), and in a similar vein, Question 2 of \cite{EM} asks whether a MS-measurable $\omega$-categorical structure is necessarily one-based. For both of these questions the key examples to be considered are Hrushovski's non locally modular supersimple $\omega$-categorical structures from \cite{Hr97} (and \cite{Hrpp}). In this section we apply Corollary \ref{23}  to show that \textit{some of} Hrushovski's examples are \textit{not} MS-measurable.  In particular, this answers Question 7 of \cite{EM}: there is a supersimple, finite rank unimodular theory (even, $\omega$-categorical, $SU$-rank 1) which is not MS-measurable.

\subsection{The Hrushovski construction for $\omega$-categorical structures} 

We recall briefly some details of the construction method. The original version of this is in \cite{Hrpp}, where it is used to provide a counterexample to Lachlan's conjecture, and \cite{Hr97}, where it is used to construct a non-modular, supersimple $\aleph_0$-categorical structure. The book \cite{W} is a very convenient reference for this (see Section 6.2.1). Generalisations and reworkings of the method (particularly relating to simple theories) are also to be found in \cite{E:predim}. We will restrict to the simplest form of the construction appropriate for producing $\omega$-categorical structures of $SU$-rank 1.

\def\KB{\overline{\mathcal{K}}}
\def\K{\mathcal{K}}

We work with a finite relational language $L = \{R_i: i \leq m\}$. For later use, it will be convenient to assume that this contains some $3$-ary relation $R$.   Recall that if $B, C$ are $L$-structures with a common substructure $A$ then the \textit{free amalgam} $B\coprod_A C$ of $B$ and $C$ over $A$ is the $L$-structure whose domain is the disjoint union of $B$ and $C$ over $A$ and whose atomic relations are precisely those of $B$ together with those of $C$. Let $\KB$ be the class of $L$-structures  and  denote by $\K$ the finite structures in $\KB$. 

For $A \in \K$ define the \textit{predmension} $\delta(A) = \vert A \vert - \sum_i\vert R_i[A]\vert$.  If $A \subseteq B \in \K$ write $A \leq B$ to mean $\delta(A) < \delta(B')$ for all $A \subset B' \subseteq B$. (We sometimes say that $A$ is \textit{self-sufficient} in $B$.) For structures in $\K$, one has:

$\bullet$ If $X \subseteq B$ and $A \leq B$, then $X\cap A \leq X$;\newline
$\bullet$ If $A \leq B \leq C$, then $A \leq C$.

Consequently (cf. \cite{W}, Corollary 6.2.8), for each $B \in \K$ there is a closure operation given by 
$\cl_B(X) = \bigcap \{ A : X \subseteq A \leq B\} \leq B$ for $X \subseteq B$. Of course, if $B \leq C \in \K$ and $X \subseteq B$, then $\cl_B(X) = \cl_C(X)$. 

The relation $\leq$ can be extended to infinite structures so that  the above  properties still hold: if $M \in \KB$ and $A \subseteq M$, write $A \leq M$ to mean that $A\cap X \leq X$ for all finite $X \subseteq M$. 

If $A, B \in \KB$, an embedding $\alpha : A \to B$ with $\alpha(A) \leq B$ is referred to as a $\leq$-embedding.

Now consider $\KB_0$, the class of $B \in \KB$ with $\emptyset \leq B$. Equivalently, if $A \subseteq B$ is finite and non-empty, then $\delta(A) > 0$. Let $\K_0$ be the finite structures in $\KB_0$.  Any structure $B$ in $\KB_0$ carries a notion of \textit{dimension} $d^B$ associated to the predimension $\delta$ and a notion of \textit{$d^B$-independence}. If $X, Y \subseteq B$ are finite, write $d^B(X) = \delta(\cl_B(X)) = \min(\delta(Y) : X \subseteq Y \subseteq B)$ and $d^B(X/Y) = d^B(X\cup Y) - d^B(Y)$. If the ambient structure $B$ is clear from the context  then we omit it from the notation. Say that finite $X, Z$ are \textit{$d$-independent} over $Y$ (in $B$) if $d^B(X/YZ) = d^B(X/Y)$. In particular this implies  $\cl_B(XY)\cap \cl_B(YZ) = \cl_B(Y)$. (Here, we use the usual shorthand of $YZ$ for $Y \cup Z$.) For the particular predimension which we have given, it can be shown that $\cl_B$ satisfies the exchange condition, and therefore gives a pregeometry; furthermore, $d^B$ is the dimension in this pregeometry. 

We look at a version of the construction (also from \cite{Hr97}) where closure is uniformly locally finite. For this, we have a continuous, increasing $f : \R^{\geq 0} \to \R$ with $f(x) \to \infty$ as $x \to \infty$ and we consider $\K_f = \{A \in \K_0 : \delta(X) \geq f(\vert X \vert)\,\, \forall X \subseteq A\}$. For suitable choice of $f$ (call these \textit{good} $f$), $(\K_f, \leq)$ has the free $\leq$-amalgamation property: if $A_0 \leq A_1, A_2 \in \K_f$, then $A_i \leq A_1\coprod_{A_0} A_2 \in \K_f$. In this case we have an associated \textit{generic structure} $M_f$ (cf. \cite{W}, Theorem 6.2.13). This is a countable structure characterised by the following properties:
\begin{enumerate}
\item $M_f$ is the union of a chain of finite self-sufficient substructures, all in $\K_f$.
\item (\textit{$\leq$-Extension Property}) If $A \leq M_f$ is finite and $A \leq B \in \K_f$, there is a $\leq$-embedding $\beta : B \to M_f$ with $\beta(a) = a$ for all $a \in A$.
\end{enumerate}
Equivalently, $\K_f$ is the class of finite substructures of $M_f$, and isomorphisms between finite self-sufficient substructures of $M_f$ extend to automorphisms of $M_f$ (we refer to the latter property as $\leq$-homogeneity). Because of the function $f$, closure in $M_f$ is uniformly locally finite and (using free amalgamation and the $\leq$-extension property) it is equal to algebraic closure (\cite{W}, Lemma 6.2.17). It then follows from $\leq$-homogeneity that   $M_f$ is $\omega$-categorical and the type of a tuple is determined by the isomorphism type of its closure.

\begin{remarks}\label{frem}\rm
To construct good functions, we can take $f$ which are piecewise smooth,  and  where the right derivative $f'$ satisfies $f'(x) \leq 1/(x+1)$  and is non-increasing. The latter condition implies that $f(x+y) \leq f(x) + yf'(x)$ (for $y \geq 0$). It can be shown that under these conditions, $\K_f$ has the free $\leq$-amalgamation property. (This is originally from \cite{Hrpp}; see also Example 6.2.27 of \cite{W}, or Lemma 3.3 of \cite{E:predim}.)
\end{remarks}

\begin{remarks} \label{fremsimple} \rm (The following  is from \cite{Hr97}; see also Example 6.2.27 of \cite{W} or Corollary 2.24 and Theorem 3.6 of \cite{E:predim}.)
If $f$ also satisfies the slower growth condition 
\[ f(3x) \leq f(x) + 1\]
then the structure $M_f$ is supersimple of $SU$-rank 1. Moreover, for tuples $\a, \b$ in $M_f$, we have $SU(\tp(\b/ \a)) = d(\b/\a)$. To see the latter, note that (by additivity of both sides) it suffices to prove this when $\b$ is a single element $b$. Now, $d(b/\a)$ is a natural number and at most $\delta(b)$, so is $0$ or $1$. If it is $0$, then $b \in \acl(\a)$ so $SU(b/\a) = 0$. Thus, it suffices to show that if $\tp(b/\a)$ divides over $\emptyset$, then $d(b/\a) < d(b/\emptyset)$. This is done (in greater generality)  in the above references.
\end{remarks}

\subsection{The dimension function}

For the rest of the section suppose that $f$ is a good function as in Remarks \ref{frem} and $M_f$ is the corresponding generic structure. We suppose that $h = (\dim, \m) : \Def(M_f) \to \R^{> 0}$ is a dimension - measure function. In this subsection we relate $\dim$ to the dimension $d$ coming from the predimension (which will be the same as $SU$-rank if $M_f$ is simple), and the measure will not be used.

\begin{notation}\rm For tuples $\a, \b$ in $M_f$ let $\loc(\b/\a)$, the\textit{ locus of $\b$ over $\a$}, be the set of realizations in $M_f$ of $\tp_L(\b/ \a)$, the $L$-type of $\b$ over $\a$. By $\omega$-categoricity, this is definable by an $L$-formula with parameters from $\a$. Let $\dim(\b/\a)$  denote the dimension  of this set.
\end{notation}

The Fubini property in MS-measurability implies that $\dim$ is additive: $\dim(\b/\a) = \dim(\a\b/ \emptyset)-\dim(\a/\emptyset)$. We also have $\dim(M_f^n) = n \dim(M_f)$. Note the existence of $\dim$-generic points: if $D \in \Def(M_f)$ is definable over a finite tuple $\a$, then  $\dim(D) = \max (\dim(\b/\a): \b \in D)$. From this we deduce that if $D' \subseteq D$ is definable, then $\dim(D') \leq \dim(D)$. A further property of $\dim$ which we require is the weak algebraicity property  that if $\b \in \acl(\a)$, then $\dim(\b/\a) = 0$. Of course, $d$ also has these properties. 

Under these assumptions on $\dim$ (and the given conditions on $f$) we will show that $\dim$ is just a scaled version of the dimension $d$. 

\begin{theorem} \label{dimthm} Suppose $f'(x) \leq \frac{1}{2(x+1)}$. If $\a, \b$ are finite tuples in $M_f$, then we have
\[\dim(\b/\a) = \dim(M_f)d(\b/\a).\]
\end{theorem}

The theorem follows from the following (always assuming the given condition on $f$).

\begin{proposition}\label{nonorthog} Let $\a, b \in M_f$ with $b\not\in \acl(\a)$ and $P = \loc(b/\a)$.  Then for every $r \in \N$ and  $\y \in M^r$ there is some $\x \in P^{r+2}$ with $\y \in \acl(\x\a)$.
\end{proposition}

We note that  \cite{Marimon} shows that  Theorem  \ref{dimthm} holds for a wider class of Hrushovski constructions than we give here.

First we show how Theorem \ref{dimthm} follows from the proposition.

\begin{proof}[Proof of Theorem \ref{dimthm}]
By the additivity  property of both $\dim$ and $d$, it will suffice to prove the statement when $\b= b$ is a single element. If $b \in \acl(\a)$, then the statement holds as both sides of the equation are zero, by the weak algebraicity property of $\dim$ and $d$. So now suppose that  $b \not\in \acl(\a)$. Let $P = \loc(b/\a)$, as in Proposition \ref{nonorthog}.  Consider 
\[Y =  \{\y = (y_1,\ldots, y_r)  \in M_f^r : y_1,\ldots, y_r \in \acl(\x\a)  \mbox{ for some } \x \in P^{2+r}\}.\]
By $\omega$-categoricity, this set is definable by an $L$-formula with parameters from $\a$ (for example, it is invariant under automorphisms of $M_f$ fixing $\a$). Thus (by existence of generic points for $\dim$) there is $\c \in Y$ with $\dim(Y) = \dim(\c/\a)$. By definition of $Y$, there are $b_1,\ldots, b_{r+2} \in P$ with $\c \in \acl(\a b_1\ldots b_{r+2})$. It follows (using weak algebraicity) that 
\[ \dim(Y) = \dim(\c/\a) \leq \dim(b_1\ldots b_{r+2}/\a)  \leq \dim(M_f^{r+2}) = (r+2)\dim(M_f).\]
But by Proposition \ref{nonorthog}, we have $Y = M_f^r$. So 
\[r\dim(M_f) = \dim(M_f^r)  = \dim(Y)  \leq (r+2)\dim(M_f).\]
Dividing by $(r+2)$ and letting $r \to \infty$, we obtain that $\dim(b/\a) = \dim(M_f)$. As $d(b/\a) = 1$, this gives $\dim(b/\a) = \dim(M_f)d(b/\a)$, as required. 
\end{proof}

The proof of Proposition \ref{nonorthog} is a  technical argument with Hrushovski constructions, so we relegate it to a separate Section: see \ref{messy_proof}. Marimon's approach \cite{Marimon} to proving non-MS-measurability of other examples of $\omega$-categorical Hrushovski constructions avoids the need for a result such as Theorem \ref{dimthm}. 

\begin{remarks} \label{Psf} \rm  It is an open problem to determine whether any of the $M_f$ are (or are not) pseudofinite. We note that Theorem \ref{dimthm} provides some information relevant to this question. 
Suppose that $f$ is a good function with $f'(x) \leq \frac{1}{2(x+1)}$ and $\K_f$ is the corresponding amalgamation class with generic structure $M_f$. \textit{Assume} that $M_f$ is elementarily equivalent to an ultraproduct $\M = \Pi_U F_i$ of finite structures. Following \cite{Hr:psfdim}, if $\Phi(\x)$ is a formula with parameters from $M$, then  the \textit{coarse pseudofinite dimension} $\Delta(\Phi(\x))$ is the standard part of the non-standard real $\Pi_U \log\vert \Phi(F_i)\vert / \log\vert F_i\vert$. We will show that for every $L$-formula $\Phi(\x)$ (without parameters), we have $\Delta(\Phi(\x)) = d(\Phi(\x))$. 

In principle, we could deduce the result from Theorem \ref{dimthm} as $\Delta$ has the properties required in the proof of \ref{dimthm}, as long as we expand the language by dimension quantifiers so that it becomes continuous (see Section 2.7 of \cite{Hr:psfdim}). However, it seems clearer to give a fuller  argument which  is essentially  a modification of that given for Theorem \ref{dimthm}.  

If $\a$ is a finite tuple in $M_f$, let $\Phi_{\a}(\x)$ denote an $L$-formula isolating $\tp(\a/\emptyset)$ (the $L$-type of $\a$ in $M_f$). Such a formula exists, by $\omega$-categoricity. If $\b$ is another tuple, then $\Phi_{\a\b}(\a, \y)$ isolates $\tp(\b/\a)$. 

\textit{Claim:\/} Suppose $\a$ is a $k$-tuple in $M_f$ and $b \in M_f$. Suppose  $\u$ is a $k$-tuple in $\M$ and $\M \models \Phi_{\a}(\u)$. Then $\Delta(\Phi_{\a b}(\u, y)) = d(b/\a)$. 

If $d(b/\a) = 0$ then $b$ is algebraic over $\a$. The size of $\acl(\a)$ is bounded uniformly (actually, in $k$),  so $\Phi_{\a b}(\u, y)$  has finitely many solutions in $\M$. Thus its pseudofinite dimension is 0. 

Now suppose $b \not\in \acl(\a)$, so $d(b/\a) = 1$. Let $r \in \N$. There is a formula $C_r(y, x_1\ldots x_{r+2}\z)$ such that if $\Phi_{\a b}(\a' b_i)$ (for $i \leq r+2$), then $C_r(M_f, b_1\ldots b_{r+2}, \a')$ is $\acl(b_1\ldots b_{r+2}, \a')$. Let $K(r)$ bound the size of this algebraic closure.

The set $Y$ in the proof of Theorem \ref{dimthm} is defined by $Y(\y,\a)$ where $Y(\y, \z)$ is the formula:

\[  \exists x_1 \ldots x_{r+2} \bigwedge_{i \leq r+2} \Phi_{\a b}(\z x_i) \wedge \bigwedge_{j \leq r}C_r(y_j, x_1\ldots x_{r+2}\z).\]

Moreover, by Proposition \ref{nonorthog}, 
\[ M_f \models (\forall \z)( \Phi_{\a}(\z) \to \forall y_1\dots y_r Y(y_1\ldots y_r \z)),\]
thus this formula also holds in $\M$.

Suppose $\u \in \M$ and $\M \models \Phi_{\a}(\u)$. Denote by $\u_i$ the $k$-tuple of $i$-th coordinates (in $F_i$) in $\u$. From the definition of $Y$, for almost all $i$ we have:

\[ \vert Y(F_i, \u_i) \vert \leq K(r)^r \vert \Phi_{\a b}(\u_i, F_i) \vert^{r+2}.\]

Thus, as $Y(\M, \u) = M^r$, for almost all $i$:

\[ K(r)^r \vert \Phi_{\a b}(\u_i, F_i) \vert^{r+2} \geq \vert F_i \vert^r.\]

As $\vert F_i \vert \to \infty$, we obtain $\Delta(\Phi_{\a b}(\u, y)) \geq \frac{r}{r+2}$. But $r$ here is arbitrary, therefore $\Delta(\Phi_{\a b}(\u, y)) \geq 1$. The reverse inequality is trivial, so we have the claim.

Now suppose $\b = (b_1\ldots b_n)$ is an $n$-tuple in $M_f$. We show that if $\u$ is a  tuple in $\M$ and $\M \models \Phi_{\a}(\u)$, then $\Delta(\Phi_{\a\b}(\u, \y)) = d(\b/\a)$. The required formula for general $L$-definable sets follows as each such is a finite  union of pairwise-disjoint sets of this form. We may assume that $d(\b/\a) = n$ and  we prove the result by induction on $n$. Let $D$ be $\Phi_{\a\b}(\u, \M) \subseteq \M^n$ and $E = \Phi_{\a b_1\ldots b_{n-1}}(\u, \M)$. Let $f : D \to E$ be projection onto the first $n-1$ coordinates. By the claim, the fibres of $f$ have coarse pseudofinite dimension $d(b_n/b_1\ldots b_{n-1} \a) = 1$. By induction hypothesis, $\Delta(E) = n-1$. Thus, by Lemma 2.8 (4) of \cite{Hr:psfdim}, $\Delta(D) = n-1 +1 = n$, as required.  (In order to apply the results from \cite{Hr:psfdim} we need to first enrich the language so that $\Delta$ becomes continuous, but this has no effect on the dimension of  formulas in the original language.) \end{remarks}

\subsection{A structure which not MS-measurable}

\begin{theorem}\label{HrCon}
There is an $\omega$-categorical, supersimple structure $M_f$ of $SU$-rank 1 which does not satisfy the amalgamation property in  Corollary \ref{23}. In particular, $M_f$ is not MS-measurable.
\end{theorem}

\begin{proof} We choose $f$ so that: $\K_f$ is a free amalgamation class; the generic $M_f$ is supersimple of SU-rank 1; the independent amalgamation property Corollary \ref{23} does not hold. We are only interested in providing an example, so we choose economy of effort over elegance.

Take $L$ to have a $3$-ary relation $R$, a $10$-ary relation $S$ and a $11$-ary relation $U$. Let $f(x) = \log_8(x+1)$. Then $f'(x) = 1/(\ln(8)(x+1)) < 1/(2(x+1))$, so by Remarks \ref{frem}, $\K_f$ is a free amalgamation class and the hypothesis on $f$ in Theorem \ref{dimthm} holds. We also have $f(3x) \leq f(x) + 1$, so by Remarks \ref{fremsimple}, the generic $M_f$ is supersimple, with $d$-independence being the same as non-forking, and  $M_f$ is of $SU$-rank 1.

Consider the $L$-structure $A$ with points $a_1,\ldots, a_{10}, u_1, \ldots, u_r$, where $r = 8^9 - 11$, and relations $S(a_1,\ldots,a_{10})$ and $U(a_1,\ldots,a_{10}, u_i)$ (for $i \leq r$). Then $\delta(A) = 9$ and $\vert A \vert = 8^9 - 1$, so $\delta(A) \geq f(\vert A \vert)$. It is easy to check that for any $X \subset A$ we have $\delta(X) \geq f(\vert X\vert)$, so $A \in \K_f$. Moreover (in the notation of \ref{23}) for each $I \in [10]^9$, the tuple $\a_I$ is $d$-independent (in $A$) and has closure $A$. Note also that if $I \in [10]^8$ then $\a_I \leq A$. 

Suppose, for a contradiction, that the conclusion of Corollary \ref{23} holds, where $\dim$ is given by $SU$-rank (in this case, given by the dimension function $d$).  We will apply this where $n = 10$, $S = M_f^{10}$ and 
\[ E = \{ \alpha(a_1\ldots a_{10}) : \alpha : A \to M_f \mbox{ is an $\leq$-embedding}\}.\]
Note that $E$ is $\emptyset$-definable, the algebraic closure (equal to the $\leq$-closure) of every element of $E$ is isomorphic to $A$ and (by the $\leq$-homogeneity of $M_f$) all elements of $E$ have the same type over $\emptyset$. 

So if the conclusion of Corollary \ref{23} holds, there exist a $d$-independent set $B_0 = \{b_1,\ldots, b_{10}\}$ of distinct elements of  $M_f$ with the property that for each $I \in [10]^9$ we have $\acl_{M_f}(\B_I) \cong A$ (via an isomorphism taking $\b_i \mapsto \a_i$), where $B_I = \{b_i : i \in I\}$. Let $B = \acl(B_0)$. By the $d$-independence,  $\delta(B) = 10$ and we have $\acl(B_I) \cap \acl(B_{I'}) = B_I\cap B_{I'} = B_{I\cap I'}$ for $I\neq I' \in [10]^9$.

Thus
\begin{multline*}
\vert B \vert \geq \vert B_0\vert + \vert \bigcup_{I \in [10]^9} \acl(B_I) \setminus B_0\vert 
= \vert B_0\vert  + \sum_{I \in [10]^9} \vert \acl(B_I) \setminus B_0\vert \\ \geq 10 + 10(8^9 - 1 - 9) = 10.8^9 - 90.
\end{multline*}

So 
\[ f(\vert B \vert) \geq \log_8(10. 8^9 - 89) > 10 = \delta(B)\]
therefore $B \not\in \K_f$: a contradiction. Thus the amalgamation property in the conclusion of Corollary \ref{23} does not hold, and in particular, $M_f$ is not MS-measurable. 
\end{proof}

\subsection{Proof of Proposition \ref{nonorthog}}\label{messy_proof}

Before proving the proposition, we give the following technical lemma.

\begin{lemma}\label{tech}
Suppose  $R$ is a $3$-ary relation in $L$ and $f'(x) \leq 1/2(x+1)$. Let $A \leq C, T \in \K_f$ (with $A \neq C, T$) and $E$ the free amalgam of $C$ and $T$ over $A$. Suppose $t_1, \ldots, t_r \in T\setminus A$ are $d$-independent over $A$, and let $c \in C\setminus A$. Let $F = E \cup \{s_1, \ldots, s_r\}$ with additional relations $R(c, s_i, t_i)$ (for $1 \leq i \leq r$). Then $As_1\ldots s_r, C, T \leq F$ and $F \in \K_f$. 
\end{lemma}

\begin{proof} Suppose $C \subset V \subseteq F$. If $V\cap T = A$ then (by construction) $\delta(V) = \delta(C) + \vert V\setminus C\vert$; if $V\cap T \supset A$ then $\delta(V) \geq \delta(C) + \delta(V\cap T) - \delta(A) > \delta(C)$. In either case, $\delta(V) > \delta(C)$, so $C \leq F$. A similar argument shows $T \leq F$. 

By free amalgamation, it is enough to prove the rest of the lemma in the case where $T = \cl_T(At_1\ldots t_r)$ and $C = \cl_C(Ac)$. So henceforth assume this.  Suppose $As_1\ldots s_r \subset V \subseteq F$ has $\delta(V)\leq \delta(As_1\ldots s_r) = \delta(A)+r$.  We can assume that $V \leq F$. Clearly $c \in V$ and therefore $t_1, \ldots, t_r \in V$. It follows that $V = F$. But $\delta(F) = \delta(A) + r + 1$, a contradiction. 

Finally we show that $F \in \K_f$. Let $X \subseteq F$. We need to show $\delta(X) \geq f(\vert X\vert)$. As $X\cap (T\cup C)$ is  the free amalgamation of $X\cap T$ and $X \cap C$ over $X \cap A$, the structure $X$ is of the same form as $F$ (possibly together with some points $s_i$ not lying in any relation in $X$). So it will suffice to prove that $\delta(F) \geq f(\vert F\vert)$. 

\smallskip

\noindent\textit{Case 1:\/} Suppose $\vert T\setminus A\vert \leq r\vert C \setminus A \vert$. 

Note that $\vert F\vert = \vert C \vert+ \vert T\setminus A\vert +r$ and $\delta(F) = \delta(C) + r$. As $C \in \K_f$ we have $\delta(C) \geq f(\vert C\vert)$. Furthermore, as the graph of $f$ lies below its tangent at any point, and $f'(x) \leq 1/2(x+1)\leq 1/(x+1)$ we have
\begin{multline*}
f(\vert F \vert) \leq f(\vert C \vert) + (\vert T\setminus A\vert +r)f'(\vert C \vert) \leq \\
f(\vert C \vert)+  \frac{1}{(\vert C \vert+1)}r (\vert C\setminus A\vert +1) \leq \delta(C)+r = \delta(F)\end{multline*}
as required.

\smallskip

\noindent\textit{Case 2:\/} Suppose $\vert T\setminus A\vert \geq r\vert C \setminus A \vert$.

This is similar. We have  $\vert F\vert = \vert T \vert+ \vert C\setminus A\vert +r$ and $\delta(F) = \delta(T) + 1$. Then 
\begin{multline*}
f(\vert F \vert) \leq f(\vert T \vert) + (\vert C\setminus A\vert +r)f'(\vert T \vert) \leq \\
f(\vert T \vert)+  \frac{1}{2\vert T \vert} (\vert C\setminus A\vert +r) \leq \delta(T)+1 = \delta(F),\end{multline*}
using the fact that $\vert T\setminus A\vert \geq \vert C\setminus A\vert, r$. 
\end{proof}

\begin{proof}[Proof of Proposition \ref{nonorthog}]  Recall that we are assuming that the language $L$ contains a $3$-ary relation symbol $R$, so we can use the previous lemma.  Let $A = \acl(\a)$ and  $B = \acl(Ab)$.

First, we note that it is enough to prove the proposition in the case where $\y$ is $d$-independent over $\a$ (that is, $d(\y/\a) = r$). To see this, take $\y_1 \subseteq \y$ which is $d$-independent over $\a$ and has $\y \in \acl(\y_1\a)$; extend this to an $r$-tuple $\y'$ which is $d$-independent over $\a$. If  $\x \in P^{r+2}$ has $\y_1 \in \acl(\a\x)$, then $\y \in \acl(\a\x)$.

\noindent\textit{Step 1:\/} We first assume that $\y = (s_1,\ldots, s_r)$ is $d$-independent over $\a$ and $A\y \leq M_f$. We shall show that there is $(b_0,\ldots, b_r) \in P^{r+1}$ with $\y \in \acl(\a, b_0,\ldots, b_r)$. 

We apply Lemma \ref{tech} with $T$  the free amalgam of $r$ copies $B_j$ ($1 \leq j \leq r$) of $B$ over $A$ and $C$ another copy of $B$. Let $b_1,\ldots, b_r, b_0$ be the corresponding copies of $b$ (over $A$) inside $B_1, \ldots, B_r, C$ respectively. Let $F$ be the disjoint union over $A$ of $A\y$, $C$ and  $T$, but with the extra relations $R( b_j, s_j, b_0)$, where $1 \leq j \leq r$, as in the lemma. Then by the lemma,
\begin{enumerate}
\item $A\y \leq F$;
\item $B_j \leq F$; and
\item $F \in \K_f$.
\end{enumerate}

Then by (i), (iii) and the $\leq$-extension property  we can assume $F \leq M_f$; by (ii), we then have  $\x = (b_0,b_1,\ldots, b_r) \in P^{r+1}$; then, because of  the relations $R(b_j, s_j, b_0)$ we have $s_j \in \acl(b_0, b_j, A)$, so $\y \in \acl(\a\x)$, as required.

\noindent\textit{Step 2:\/} Now let $\y = (t_1,\ldots, t_r)$ be $d$-independent over $A$ and let $T = \acl(A\y)$. Let $C$ be a copy of $B$ over $A$ with $c$ the copy of $b$ over $A$ inside $C$, and let $F$ be constructed as in the lemma. As in step 1, we can assume that $F \leq M_f$. So $c \in P$ and  $\y \in \acl(\a, c, s_1,\ldots, s_r)$. But by step 1 (and $As_1\ldots s_r \leq F$) the tuple $(s_1,\ldots, s_r)$ is in $\acl(\a\z)$ for some $\z \in P^{r+1}$. The result follows.
\end{proof}

\end{document}